\documentclass[leqno,11pt]{amsart}\usepackage{psfrag}
\usepackage[dvips]{color}
\usepackage{multicol}
\usepackage{picinpar}
\usepackage{enumerate}
\usepackage[hyphens]{url}        %paquete para el hypervinculo con el indice, referencias, etc.
\usepackage[breaklinks,colorlinks=true,linkcolor=blue,citecolor=blue, urlcolor=blue]{hyperref}                                                 %configuracion para el hypervinculo
\usepackage[utf8x]{inputenc}
\usepackage{txfonts}
\usepackage{graphicx}
\usepackage[T1]{fontenc}

\usepackage{amsmath,amstext,amssymb,amsopn,amsthm,mathrsfs}
\theoremstyle{plain}
\newtheorem{teo}{Theorem}[section]
\newtheorem{defi}{Definition}[section]
\newtheorem{cor}{Corollary}[section]
\newtheorem{lem}{Lemma}[section]
\newtheorem{prop}{Proposition}[section]
\newtheorem{obs}{Observation}[section]
\numberwithin{equation}{section}
\newtheorem{ejm}{Example}[section]

\begin{document}

\title[Gaussian Bessel Potentials \& Bessel Fractional Derivatives.] {Boundedness of Gaussian Bessel Potentials and Bessel Fractional Derivatives on variable Gaussian Besov-Lipschitz spaces.}

\author{Ebner Pineda}
\address{Escuela Superior Polit\'ecnica del Litoral. ESPOL, FCNM, Campus Gustavo Galindo Km. 30.5 V\'ia Perimetral, P.O. Box 09-01-5863, Guayaquil, ECUADOR.}
\email{epineda@espol.edu.ec}
\author{Luz Rodriguez}
\address{Escuela Superior Polit\'ecnica del Litoral. ESPOL, FCNM, Campus Gustavo Galindo Km. 30.5 V\'ia Perimetral, P.O. Box 09-01-5863, Guayaquil, ECUADOR.}
\email{luzeurod@espol.edu.ec}
\author{Wilfredo~O.~Urbina}
\address{Department of Mathematics, Actuarial Sciences and Economics, Roosevelt University, Chicago, IL,
   60605, USA.}
\email{wurbinaromero@roosevelt.edu}



\begin{abstract}
In this paper we study the regularity properties of the Gaussian Bessel potentials and Gaussian Bessel fractional derivatives on variable Gaussian Besov-Lipschitz spaces $B_{p(\cdot),q(\cdot)}^{\alpha}(\gamma_{d}),$ that were defined in a previous paper \cite{Pinrodurb}, under certain conditions on $p(\cdot)$ and $q(\cdot)$.
\end{abstract}

\maketitle

\section{Introduction and Preliminaries}

On $\mathbb{R}^d$ let us consider the Gaussian measure
\begin{equation}
\gamma_d(x)=\frac{e^{-\left|x\right|^2}
}{\pi^{d/2}} dx, \, x\in{\mathbb R}^d
\end{equation}
 and the Ornstein-Uhlenbeck
differential operator
\begin{equation}\label{OUop}
L=\frac12\triangle_x-\left\langle x,\nabla _x\right\rangle.
\end{equation}

Let $\nu=(\nu _1,...,\nu_d)$ be a
multi-index such that $\nu _i \geq 0, i= 1, \cdots, d$,  let $\nu
!=\prod_{i=1}^d\nu _i!,$ $\left| \nu \right| =\sum_{i=1}^d\nu _i,$ $%
\partial _i=\frac \partial {\partial x_i},$ for each $1\leq i\leq d$ and $%
\partial ^\nu =\partial _1^{\nu _1}...\partial _d^{\nu _d},$ consider the normalized Hermite polynomials of order $\nu$ in $d$ variables,
\begin{equation}
h_\nu (x)=\frac 1{\left( 2^{\left| \beta \right| }\nu
!\right)
^{1/2}}\prod_{i=1}^d(-1)^{\nu _i}e^{x_i^2}\frac{\partial ^{\nu _i}}{%
\partial x_i^{\nu _i}}(e^{-x_i^2}),
\end{equation}
it is well known, that the Hermite polynomials are
eigenfunctions of the operator $L$,
\begin{equation}\label{eigen}
L h_{\nu}(x)=-\left|\nu \right|h_\nu(x).
\end{equation}
Given a function $f$ $\in L^1(\gamma _d)$ its
$\nu$-Fourier-Hermite coefficient is defined by
\[
\hat{f}(\nu) =<f, h_\nu>_{\gamma_d}
=\int_{{\mathbb R}^d}f(x)h_\nu (x)\gamma _d(dx).
\]
Let $C_n$ be the closed subspace of $L^2(\gamma_d)$ generated by
the linear combinations of $\left\{ h_\nu \ :\left| \nu
\right| =n\right\}$. By the orthogonality of the Hermite
polynomials with respect to $\gamma_d$ it is easy to see that
$\{C_n\}$ is an orthogonal decomposition of $L^2(\gamma_d)$,
$$ L^2(\gamma_d) = \bigoplus_{n=0}^{\infty} C_n,$$
this decomposition is called the Wiener chaos.

Let $J_n$ be the orthogonal projection  of $L^2(\gamma_d)$ onto
$C_n$, then if $f\in L^2(\gamma_d)$
\[
J_n f=\sum_{\left|\nu\right|=n}\hat{f}(\nu) h_\nu.
\]
Let us define the Ornstein-Uhlenbeck semigroup $\left\{
T_t\right\} _{t\geq 0}$ as
\begin{eqnarray}\label{01}
\nonumber T_t f(x)&=&\frac 1{\left( 1-e^{-2t}\right) ^{d/2}}\int_{{\mathbb R}^d}e^{-\frac{%
e^{-2t}(\left| x\right| ^2+\left| y\right| ^2)-2e^{-t}\left\langle
x,y\right\rangle }{1-e^{-2t}}}f(y)\gamma _d(dy)\\
& = & \frac{1}{\pi^{d/2}(1-e^{-2t})^{d/2}}\int_{\mathbb R^d} e^{-
\frac{|y-e^{-t}x|^2}{1-e^{-2t}}} f(y) dy
\end{eqnarray}
The family $\left\{ T_t\right\}_{t\geq 0}$ is a strongly
continuous Markov semigroup on $L^p(\gamma_d)$, $1 \leq p\leq
\infty$, with infinitesimal generator $L$. Also, by a change of
variable we can write,
\begin{equation}\label{t1}
T_t f(x)=\int_{{\mathbb R}^d} f(\sqrt{1-e^{-2t}}u + e^{-t}x)\gamma
_d(du).
\end{equation}

Now, by Bochner subordination formula, see Stein \cite{st1} page 61, we
define the Poisson-Hermite semigroup $\left\{ P_t\right\} _{t\geq
0}$ as
\begin{equation}\label{PoissonH}
 P_t f(x)=\frac 1{\sqrt{\pi }}\int_0^{\infty} \frac{e^{-u}}{\sqrt{u}}T_{t^2/4u}f(x)du
\end{equation}
From (\ref{01}) we
obtain, after the change of variable $r=e^{-t^2/4u}$,
\begin{eqnarray}\label{03}
\nonumber P_t f(x)&=&\frac 1{2\pi
^{(d+1)/2}}\int_{{\mathbb R}^d}\int_0^1t\frac{\exp \left( t^2/4\log
r\right) }{(-\log r)^{3/2}}\frac{\exp \left( \frac{-\left|
y-rx\right| ^2}{1-r^2}\right) }{(1-r^2)^{d/2}}\frac{dr}rf(y)dy\\
&=& \int_{{\mathbb R}^d} p(t,x,y) f(y)dy,
\end{eqnarray}
with
\begin{equation}
p(t,x,y) = \frac 1{2\pi ^{(d+1)/2}}\int_0^1t\frac{\exp \left(
t^2/4\log r\right) }{(-\log r)^{3/2}}\frac{\exp \left(
\frac{-\left| y-rx\right| ^2}{1-r^2}\right)
}{(1-r^2)^{d/2}}\frac{dr}r.
\end{equation}
Also by the change of variables $s= t^2/4u$ we have,
\begin{equation}\label{poissonrepstable}
P_t f(x)=\frac 1{\sqrt{\pi }}\int_0^{\infty} \frac{e^{-u}}{\sqrt{u}}T_{t^2/4u}f(x)du
=\int_0^{\infty} T_s f(x) \mu^{(1/2)}_t(ds),
\end{equation}
where the measure
\begin{equation}\label{onesided1/2}
\mu^{(1/2)}_t(ds) = \frac t{2\sqrt{\pi
}}\frac{e^{-t^2/4s}}{s^{3/2}}ds,
\end{equation}
is called the one-side stable measure on $(0, \infty)$ of order
$1/2$.

The family $\left\{ P_t\right\}_{t\geq 0}$ is also a strongly
continuous semigroup on $L^p(\gamma_d)$, $1 \leq p < \infty$,
with infinitesimal generator $-(-L)^{1/2}$. In what follows, often we are  going to use the notation $$u(x,t) = P_{t}f(x),$$
and
$$ u^{(k)}(x,t) = \frac{\partial^k}{\partial t^k} P_{t}f(x).$$

Observe that by (\ref{eigen}) we have that
\begin{equation}\label{OUHerm}
T_t h_\nu(x)=e^{-t\left| \nu\right|}h_\nu(x),
\end{equation}
and
\begin{equation}\label{PHHerm}
 P_t h_\nu(x)=e^{-t\sqrt{\left| \nu\right|}}h_\nu(x),
\end{equation}
i.e. the Hermite polynomials are eigenfunctions of $T_t$ and $P_t$ for any $t \geq 0$.\\

For completeness, let us get more background on variable Lebesgue spaces with respect to a Borel measure $\mu$.

A $\mu$-measurable function $p(\cdot):\Omega\rightarrow [1,\infty]$ is an exponent function, the set of all the exponent functions will be denoted by  $\mathcal{P}(\Omega,\mu)$. For $E\subset\Omega$ we set $$p_{-}(E)=\text{ess}\inf_{x\in E}p(x) \;\text{and}\; p_{+}(E)=\text{ess}\sup_{x\in E}p(x),$$
and $\Omega_{\infty}=\lbrace x\in \Omega:p(x)=\infty\rbrace$.\\
  We use the abbreviations $p_{+}=p_{+}(\Omega)$ and $p_{-}=p_{-}(\Omega)$.

\begin{defi}\label{deflogholder}
Let $E\subset \mathbb{R}^{d}$. We say that $\alpha(\cdot):E\rightarrow\mathbb{R}$ is locally log-H\"{o}lder continuous, and denote this by $\alpha(\cdot)\in LH_{0}(E)$, if there exists a constant $C_{1}>0$ such that
			\begin{eqnarray*}
				|\alpha(x)-\alpha(y)|&\leq&\frac{C_{1}}{log(e+\frac{1}{|x-y|})}
			\end{eqnarray*}
			for all $x,y\in E$. We say that $\alpha(\cdot)$ is log-H\"{o}lder continuous at infinity with base point at $x_{0}\in \mathbb{R}^{d}$, and denote this by $\alpha(\cdot)\in LH_{\infty}(E)$, if there exist  constants $\alpha_{\infty}\in\mathbb{R}$ and $C_{2}>0$ such that
			\begin{eqnarray*}
				|\alpha(x)-\alpha_{\infty}|&\leq&\frac{C_{2}}{log(e+|x-x_{0}|)}
			\end{eqnarray*}
			for all $x\in E$. We say that $\alpha(\cdot)$ is log-H\"{o}lder continuous, and denote this by $\alpha(\cdot)\in LH(E)$ if both conditions are satisfied.
			The maximum, $\max\{C_{1},C_{2}\}$ is called the log-H\"{o}lder constant of $\alpha(\cdot)$.
\end{defi}

\begin{defi}\label{defPdlog}
			We say that $p(\cdot)\in\mathcal{P}_{d}^{log}(E)$, if $\frac{1}{p(\cdot)}$ is log-H\"{o}lder continuous and  denote by $C_{log}(p)$ or $C_{log}$ the log-H\"{o}lder constant of $\frac{1}{p(\cdot)}$.
		\end{defi}
		
\begin{defi}
For a $\mu$-measurable function $f:\mathbb{R}^{d}\rightarrow \overline{\mathbb{R}}$, we define the modular
\begin{equation}
\rho_{p(\cdot),\mu}(f)=\displaystyle\int_{\mathbb{R}^{d}\setminus\Omega_{\infty}}|f(x)|^{p(x)}\mu(dx)+\|f\|_{L^{\infty}(\Omega_{\infty},\mu)},
\end{equation}

The variable exponent Lebesgue space on $\mathbb{R}^{d}$, $L^{p(\cdot)}(\mathbb{R}^{d},\mu)$ consists on those $\mu\_$measurable functions $f$ for which there exists $\lambda>0$ such that $\rho_{p(\cdot),\mu}\left(f/\lambda\right)<\infty,$ i.e.
\begin{equation*}
L^{p(\cdot)}(\mathbb{R}^{d},\mu) =\left\{f:\mathbb{R}^{d}\to\overline{ \mathbb{R}},\; \mu\_\text{measurable and } \; \rho_{p(\cdot),\mu}\left(f/\lambda\right)<\infty, \; \text{for some} \;\lambda>0\right\}.
\end{equation*}
and the norm
\begin{equation}
\|f\|_{p(\cdot),\mu}=\inf\left\{\lambda>0:\rho_{p(\cdot),\mu}(f/\lambda)\leq 1\right\}.
\end{equation}
\end{defi}
\begin{obs} When $\mu$ is the Lebesgue measure, we write  $\rho_{p(\cdot)}$ and $\|f\|_{p(\cdot)}$ instead of $\rho_{p(\cdot),\mu}$ and $\|f\|_{p(\cdot),\mu}$ respectively.\\
\end{obs}
\begin{teo} (Norm conjugate formula)
Let  $\nu$ a complete, $\sigma$-finite measure on $\Omega$.  $p(\cdot)\in \mathcal{P}(\Omega,\nu)$, then\\
\begin{equation}\label{normaconjugada}
\frac{1}{2}\|f\|_{p(\cdot),\nu}\leq \|f\|^{'}_{p(\cdot),\nu}\leq 2\|f\|_{p(\cdot),\nu},
\end{equation}
for all $f$ $\nu$-measurable on $\Omega$,where
 $$\displaystyle\|f\|^{'}_{p(\cdot),\nu}=\sup\left\{\int_{\Omega}|f||g|d\mu:g\in L^{p'(\cdot)}(\Omega,\nu),\|g\|_{p'(\cdot),\nu}\leq 1\right\}.$$
\end{teo}
\begin{proof} See Corollary 3.2.14 in \cite{LibroDenHarjHas}.
\end{proof}
\begin{teo}
(H\"older's inequality) Let $\nu$ a complete, $\sigma$-finite measure on $\Omega$. $r(\cdot),q(\cdot)\in \mathcal{P}(\Omega,\nu)$, and define $p(\cdot)\in \mathcal{P}(\Omega,\nu)$ by $\displaystyle \frac{1}{p(x)}=\frac{1}{q(x)}+\frac{1}{r(x)}$, $\nu$ a.e. $x\in \Omega$.
Then, for all $f\in L^{q(\cdot)}(\Omega,\nu)$ and $g\in L^{r(\cdot)}(\Omega,\nu)$, $ fg\in L^{p(\cdot)}(\Omega,\nu)$ and
\begin{equation}\label{Holder generalizada}
\|fg\|_{p(\cdot),\nu}\leq 2\|f\|_{q(\cdot),\nu}\|g\|_{r(\cdot),\nu}
\end{equation}
\end{teo}
\begin{proof} See Lemma 3.2.20 in \cite{LibroDenHarjHas}.
\end{proof}
\begin{teo}
(Minkowski's integral inequality for variable Lebesgue spaces) Given $\mu$ and $\nu$ complete $\sigma$-finite measures on $X$ and $Y$ respectively, $p\in \mathcal{P}(X,\mu)$. Let $f:X\times Y\rightarrow \overline{\mathbb{R}}$ measurable with respect to the product measure on $X\times Y$, such that for almost every $y\in Y$, $f(\cdot,y)\in L^{p(\cdot)}(X,\mu)$. Then
\begin{equation}\label{integralMinkowski}
\left\|\int_{Y}f(\cdot,y)d\nu(y)\right\|_{p(\cdot),\mu}\leq C\int_{Y}\|f(\cdot,y)\|_{p(\cdot),\mu}d\nu(y)
\end{equation}
\end{teo}
\begin{proof} It is completely analogous to the proof of Corollary 2.38 in \cite{dcruz} by interchanging the Lebesgue measure for complete $\sigma$-finite measures $\mu$ and $\nu$ on $X$ and $Y$ respectively, and by using (\ref{Holder generalizada}), Fubini's theorem and then (\ref{normaconjugada}).
\end{proof}

In what follows $\mu$ represents the Haar measure $\displaystyle\mu(dt)=\frac{dt}{t}$ on $\mathbb{R}^{+}$.\\
\begin{obs}
For a $\mu$-measurable function $f:\mathbb{R}^{+}\rightarrow \overline{\mathbb{R}}$, $q(\cdot)\in\mathcal{P}(\mathbb{R}^{+},\mu)$, and any $\lambda>0$
\begin{eqnarray*}
\rho_{q(\cdot),\mu}\left(\frac{f}{\lambda}\right)&=&\displaystyle\int_{0}^{\infty}\left|\frac{f(t)}{\lambda}\right|^{q(t)}\mu(dt)=\displaystyle\int_{0}^{\infty}\left|\frac{t^{-1/q(t)}f(t)}{\lambda}\right|^{q(t)}dt\\
&=&\rho_{q(\cdot)}\left(\frac{t^{-1/q(\cdot)}f}{\lambda}\right)
\end{eqnarray*}
Thus, \begin{equation}\label{normaq,dt/t}
\|f\|_{q(\cdot),\mu}=\|t^{-1/q(\cdot)}f\|_{q(\cdot)}
\end{equation}
\end{obs}

In the case $\Omega=\mathbb{R}^{+}$, we denote by $\mathcal{M}_{0,\infty}$ the set of all measurable functions $p(\cdot): \mathbb{R}^{+}\rightarrow \mathbb{R}^{+} $ which satisfy the following conditions:
\begin{enumerate}
\item[i)] $0\leq p_{-}\leq p_{+} <\infty$.
\item[ii$_0$)] There exists $p(0)=\displaystyle\lim_{x\rightarrow 0}p(x)$ and $|p(x)- p(0)|\leq \frac{A}{\ln(1/x)}, 0< x\leq 1/2$.
\item[ii$_\infty$)] There exists  $p(\infty)=\displaystyle\lim_{x\rightarrow \infty}p(x)$ and $|p(x)- p(\infty)|\leq \frac{A}{\ln(x)},  x>2$.
\end{enumerate}

 We denote  by $\mathcal{P}_{0,\infty}$ the subset of functions $p(\cdot)$ such that $p_{-}\geq 1$.\\
 
 Let $\alpha(\cdot),\beta(\cdot)\in LH(\mathbb{R}^{+})$, bounded with \begin{equation}\label{alpha,p'}
  \displaystyle\alpha(0)<\frac{1}{p'(0)},\; \alpha(\infty)<\frac{1}{p'(\infty)}
  \end{equation} and
  \begin{equation}\label{beta,p}
  \displaystyle\beta(0)>-\frac{1}{p(0)},\; \beta(\infty)>-\frac{1}{p(\infty)}
  \end{equation}
\begin{teo}
Let $p(\cdot)\in\mathcal{P}_{0,\infty}$, $\alpha(\cdot),\beta(\cdot)\in LH(\mathbb{R}^{+})$, bounded. Then the Hardy-type inequalities
 \begin{equation}\label{hardyinequalityvariable0ax}
\left\|x^{\alpha(x)-1}\int_{0}^{x}\frac{f(y)}{y^{\alpha(y)}}dy\right\|_{p(\cdot)}\leq C_{\alpha(\cdot),p(\cdot)}\|f\|_{p(\cdot)}
 \end{equation}
  \begin{equation}\label{hardyinequalityvariablexinfty}
\left\|x^{\beta(x)}\int_{x}^{\infty}\frac{f(y)}{y^{\beta(y)+1}}dy\right\|_{p(\cdot)}\leq C_{\beta(\cdot),p(\cdot)}\|f\|_{p(\cdot)}
 \end{equation}
are valid, if and only if, $\alpha(\cdot),\beta(\cdot)$ satisfy conditions (\ref{alpha,p'}) and (\ref{beta,p})
\end{teo}
\begin{proof} For the proof see Theorem 3.1 and Remark 3.2 in \cite{dieningsamko}.
\end{proof}

As a consequence, we obtain the Hardy inequalities associated to the exponent $q(\cdot)\in\mathcal{P}_{0,\infty}$ and the measure $\mu$.
\begin{cor}
Let $q(\cdot)\in\mathcal{P}_{0,\infty}$ and $r>0$, then
\begin{equation}\label{hardyineq0atr}\displaystyle\left\|t^{-r}\int_{0}^{t}g(y)dy\right\|_{q(\cdot),\mu}\leq C_{r,q(\cdot)}\left\|y^{-r+1}g\right\|_{q(\cdot),\mu}
\end{equation} and
\begin{equation}\label{hardyineqtainftyr}\displaystyle\left\|t^{r}\int_{t}^{\infty}g(y)dy\right\|_{q(\cdot),\mu}\leq C_{r,q(\cdot)}\left\|y^{r+1}g\right\|_{q(\cdot),\mu}
\end{equation}
\end{cor}
\begin{proof} See \cite{Pinrodurb}.
\end{proof}

In what follows we also need the classical Hardy's inequalities, so for completeness we will write then here, see \cite{st1} page 272,
\begin{equation}\label{hardy1}
\int_{0}^{+\infty}\left(\int_{0}^{x}f(y)dy\right)^{p} x^{-r-1} dx \leq \frac{p}{r}\int_{0}^{+\infty}(y f(y))^{p}y^{-r-1}dy,
\end{equation}
and
\begin{equation}\label{hardy2}
\int_{0}^{+\infty}\left(\int_{x}^{\infty}f(y)dy\right)^{p} x^{r-1}dx \leq \frac{p}{r}\int_{0}^{+\infty}(y f(y))^{p}y^{r-1}dy,
\end{equation}
 where $f\geq 0, p\geq 1$ and $r>0.$\\
 
 We will consider only Lebesgue variable spaces with respect to the Gaussian measure $\gamma_d,$ $L^{p(\cdot)}(\mathbb{R}^{d},\gamma_d).$  The next condition was introduced by E. Dalmasso and R. Scotto in \cite{DalSco}.

\begin{defi}\label{defipgamma}
Let $p(\cdot)\in\mathcal{P}(\mathbb{R}^{d},\gamma_{d})$, we say that $p(\cdot)\in\mathcal{P}_{\gamma_{d}}^{\infty}(\mathbb{R}^{d})$ if there exist constants $C_{\gamma_{d}}>0$ and $p_{\infty}\geq1$ such that
\begin{equation}
   |p(x)-p_{\infty}|\leq\frac{C_{\gamma_{d}}}{\|x\|^{2}},
\end{equation}
for $x\in\mathbb{R}^{d}, x\neq \mathbf{0}$
\end{defi}
\begin{ejm}
  Consider $p(x)=p_{\infty}+\displaystyle\frac{A}{(e+\|x\|)^{q}}$, $x\in\mathbb{R}^{d}$, for any $p_{\infty}\geq 1, A\geq 0$ and $q\geq 2$ then $p(\cdot)\in\mathcal{P}_{\gamma_{d}}^{\infty}(\mathbb{R}^{d})$. 
\end{ejm}
\begin{obs}\label{obs4.1}
If $p(\cdot)\in\mathcal{P}_{\gamma_{d}}^{\infty}(\mathbb{R}^{d})$, then $p(\cdot)\in LH_{\infty}(\mathbb{R}^{d})$
\end{obs}

Additionally, we need some technical results
\begin{lem}\label{corol1} Given $k\in \mathbb{N}$ and $t>0$  then  $\mu^{(1/2)}_t$ the  one-side stable measure on $(0, \infty)$ of order
$1/2$ satisfies
\begin{equation}
 \int_{0}^{+\infty}\left| \frac{\partial^{k}\mu_{t}^{(1/2)}}{\partial
t^{k}}\right|(ds)\leq\frac{C_{k}}{t^{k}}.
\end{equation}
\end{lem}
For the proof see inequality (3.21) in \cite{urbina2019}.

\begin{lem}\label{kdecay}
Let $p(\cdot)\in\mathcal{P}_{\gamma_{d}}^{\infty}(\mathbb{R}^{d})\cap LH_{0}(\mathbb{R}^{d})$. Suppose that $f\in L^{p(\cdot)}(\gamma_d)$, then for any integer $k$,
$$\displaystyle\left\|\frac{\partial^{k}}{\partial
t^{k}}P_t f\right\|_{p(\cdot),\gamma_{d}}\leq C_{p(\cdot)}\left\|\frac{\partial^{k}}{\partial
s^{k}}P_s f\right\|_{p(\cdot),\gamma_{d}},$$
for any $0<s<t<+\infty$. Moreover,
\begin{equation}\label{kdecayine}
\left\|\frac{\partial^{k}}{\partial t^{k}}P_t f\right\|_{p(\cdot),\gamma_d}\leq
\frac{C_{k,p(\cdot)}}{t^{k}} \|f\|_{p(\cdot),\gamma_d}, \quad t>0 \quad .
\end{equation}
\end{lem}
For the proof see \cite{Pinrodurb}.\\

One of the main results in  \cite{Pinrodurb} was the definition of the variable Gaussian Besov-Lipschitz spaces $B_{p(\cdot),q(\cdot)}^{\alpha}(\gamma_d),$ following  \cite{st1} and \cite{gatPinurb}. They were defined as follows:

\begin{defi}
Let $p(\cdot)\in\mathcal{P}_{\gamma_{d}}^{\infty}(\mathbb{R}^{d})\cap LH_{0}(\mathbb{R}^{d})$ and $q(\cdot)\in\mathcal{P}_{0,\infty}$.
Let $\alpha \geq 0$, $k$ the smallest integer greater than
$\alpha$. The variable Gaussian Besov-Lipschitz space
$B_{p(\cdot),q(\cdot)}^{\alpha}(\gamma_d)$ is defined as the set of functions $f \in
L^{p(\cdot)}(\gamma_d)$ such that
\begin{equation}\label{e15}
\left\|t^{k-\alpha} \left\|
\frac{\partial^{k}P_t f}{\partial t^{k}} \right\|_{p(\cdot),\gamma_d}
\right\|_{q(\cdot),\mu}  < \infty,
\end{equation}
the norm of $f \in B_{p(\cdot),q(\cdot)}^{\alpha}(\gamma_d)$ is defined as
\begin{equation}
\left\| f \right\|_{B_{p(\cdot),q(\cdot)}^{\alpha}}: =  \left\| f \right\|_{p(\cdot),
\gamma_d} +\left\|t^{k-\alpha} \left\|
\frac{\partial^{k}   P_t f}{\partial t^{k}} \right\|_{p(\cdot),\gamma_d}\right\|_{q(\cdot),\mu}.
\end{equation}

 The  variable Gaussian Besov-Lipschitz space
$B_{p(\cdot),\infty}^{\alpha}(\gamma_d)$ is defined as the set of functions $f \in
L^{p(\cdot)}(\gamma_d)$ for which there exists a constant $A$ such that
$$\left\|\frac{\partial^{k}P_t f}{\partial
t^{k}}\right\|_{p(\cdot),\gamma_d}\leq At^{-k+\alpha}, \forall t>0$$ and then the norm of
$f \in B_{p(\cdot),\infty}^{\alpha}(\gamma_d)$ is defined as
\begin{equation}
\left\| f \right\|_{B_{p(\cdot),\infty}^{\alpha}}: =  \left\| f
\right\|_{p(\cdot),\gamma_d} +A_{k}(f),
\end{equation}
where $A_{k}(f)$ is the smallest constant $A$ in
the above inequality.

\end{defi}

For more details about the definition of  variable Gaussian Besov-Lipschitz spaces, we refer to  \cite{Pinrodurb}.\\

Additionally in \cite{Pinrodurb} we obtained some inclusion relations between variable  Gaussian 
Besov-Lipschitz spaces. These  results are analogous to Proposition 10, page 153 in \cite{st1}, see  also \cite{Pinurb} or Proposition 7.36 in \cite{urbina2019}.
\begin{prop} \label{incluBesov}
Let $p(\cdot)\in\mathcal{P}_{\gamma_{d}}^{\infty}(\mathbb{R}^{d})\cap LH_{0}(\mathbb{R}^{d})$ and $q_{1}(\cdot),q_{2}(\cdot)\in\mathcal{P}_{0,\infty}$. The inclusion $B_{p(\cdot),q_1(\cdot)}^{\alpha_{1}}(\gamma_d)\subset
B_{p(\cdot),q_2(\cdot)}^{\alpha_{2}}(\gamma_d)$ holds if:
\begin{enumerate}
\item [i)]   $\alpha_{1}>\alpha_{2}>0$ ( $q_{1}(\cdot)$ y $q_{2}(\cdot)$ not need to be related), or
\item[ii)] If $\alpha_{1}=\alpha_{2}$ and
$q_{1}(t)\leq q_{2}(t)\quad a.e.$\\
\end{enumerate}
\end{prop}

Finally, the operators that are going to be considered in this paper are the following:
\begin{itemize}

\item  The Gaussian Bessel Potential of order $\beta>0,$
$\mathcal{J}_\beta$,
is defined formally as
\begin{eqnarray}
\mathcal{J}_\beta= (I+\sqrt{-L})^{-\beta},
\end{eqnarray}
meaning that for the Hermite polynomials we have,
\begin{eqnarray*}
\mathcal{J}_\beta h_\nu(x)=\frac 1{(1+\sqrt{\left|
\nu\right|})^{\beta}}h_\nu(x).
\end{eqnarray*}
Again  by linearity can be extended to any polynomial and Meyer's theorem allows us to extend Bessel Potentials
to a \mbox{continuous} operator on $L^p(\gamma_d),$ $1 < p <
\infty$.  It can be proved that the Bessel potentials can be
represented as
\begin{equation}\label{Beselrepre}
\mathcal{J}_\beta
f(x)=\frac{1}{\Gamma(\beta)}\int_{0}^{+\infty}t^{\beta}e^{-t}P_{t}f(x)\frac{dt}{t}.
\end{equation}
Moreover, $\mathcal{J}_\beta$
is bounded on $L^{p(\cdot)}(\gamma_d)$, for $p(\cdot)\in\mathcal{P}_{\gamma_{d}}^{\infty}(\mathbb{R}^{d})\cap LH_{0}(\mathbb{R}^{d})$ with $1< p_{-}\leq p_{+} <\infty$. For the proof see \cite{MorPinUrb}.\\

\item The Gaussian Bessel fractional derivative ${\mathcal D}^\beta$,  defined formally for $\beta>0$ as
\[
{\mathcal D}^\beta=(I+\sqrt{-L})^{\beta},
\]
which means that for the Hermite polynomials, we have
\begin{equation}\label{e7}
{\mathcal D}^\beta h_\nu(x)=(1+ \sqrt{\left|
\nu\right|})^{\beta} h_\nu(x),
\end{equation}
Let $k$  be the smallest integer greater than $\beta$ i.e. $ k-1 \leq \beta < k$, then the fractional derivative ${\mathcal D}^\beta$ can be represented  as
\begin{equation}\label{derbesselrep}
{\mathcal D}^\beta f =  \frac{1}{c^k_{\beta}} \int_0^{\infty} t^{-\beta-1} (e^{-t} P_t -I )^k\,  f \, dt,
\end{equation}
where $c^k_{\beta} =\displaystyle\int_0^{\infty} u^{-\beta-1} (e^{-u} -1 )^k du.$\\
\end{itemize}

As usual in what follows $C$ represents a constant that is not necessarily the same in each occurrence.

%===================================

\section{Main results}

%============================================

The main results of the paper are the study of the regularity properties of the Gaussian Bessel potentials and the Gaussian Bessel fractional derivatives on  variable Gaussian Besov-Lipschitz spaces. \\

Let us start considering the regularity properties of the Gaussian Bessel potentials. In the following theorem we consider their action on $B_{p(\cdot),\infty}^{\alpha}(\gamma_{d})$  spaces, which is analogous to Theorem 4 in \cite{gatPinurb}.
\begin{teo}
Let $\alpha\geq 0, \beta>0$ then for $p(\cdot)\in\mathcal{P}_{\gamma_{d}}^{\infty}(\mathbb{R}^{d})\cap LH_{0}(\mathbb{R}^{d})$ with $1< p_{-}\leq p_{+} <\infty$.
Then, the Gaussian Bessel potential $\mathcal{J}_{\beta}$ is
bounded from $B_{p(\cdot),\infty}^{\alpha}(\gamma_{d})$ into
$B_{p(\cdot),\infty}^{\alpha+\beta}(\gamma_{d})$.
\end{teo}
\begin{proof}

Let $k>\alpha+\beta$ a fixed integer and $f\in
B_{p(\cdot),\infty}^{\alpha}(\gamma_{d})$, then $\mathcal{J}_{\beta}f\in L^{p(\cdot)}(\gamma_{d})$ \\(see \cite{MorPinUrb}).
By using the representation of Bessel potential (\ref{Beselrepre}) and properties of $P_{t}$, we get
$$P_{t}(\mathcal{J}_{\beta}f)(x)=\displaystyle\frac{1}{\Gamma(\beta)}\int_{0}^{+\infty}s^{\beta}e^{-s}P_{t+s}f(x)\frac{ds}{s},$$
thus using the dominated convergence theorem and chain rule, we obtain
\begin{eqnarray*}
\frac{\partial^{k}}{\partial t^{k}}P_{t}(\mathcal{J}_{\beta}f)(x)&=&\frac{1}{\Gamma(\beta)}\int_{0}^{+\infty}s^{\beta}e^{-s}u^{(k)}(x,t+s)\frac{ds}{s}.
\end{eqnarray*}
 This implies, using Minkowski's integral inequality (\ref{integralMinkowski}), that
\begin{eqnarray*}
\left\|\frac{\partial^{k}}{\partial
t^{k}} P_{t}(\mathcal{J}_{\beta}f)\right\|_{p(\cdot),\gamma_{d}}&\leq&\frac{C}{\Gamma(\beta)}\int_{0}^{+\infty}s^{\beta}e^{-s}\left\|u^{(k)}(\cdot,t+s)\right\|_{p(\cdot),\gamma_{d}}\frac{ds}{s}\\
&=& \frac{C}{\Gamma(\beta)}\int_{0}^{t}s^{\beta}e^{-s}\left\|u^{(k)}(\cdot,t+s)\right\|_{p(\cdot),\gamma_{d}}\frac{ds}{s}\\
&&\quad \quad \quad + \frac{C}{\Gamma(\beta)}\int_{t}^{\infty}s^{\beta}e^{-s}\left\|u^{(k)}(\cdot,t+s)\right\|_{p(\cdot),\gamma_{d}}\frac{ds}{s}\\
&=&(I)+(II).
\end{eqnarray*}
Now, as $\beta>0$, using Lemma \ref{kdecay} (as  $t+s>t$) and since
$f\in B_{p(\cdot),\infty}^{\alpha}(\gamma_{d})$,
\begin{eqnarray*}
(I) &\leq&\frac{C}{\Gamma(\beta)}\left\|\frac{\partial^{k}P_{t}f}{\partial
t^{k}}\right\|_{p(\cdot),\gamma_{d}}\int_{0}^{t}s^{\beta}e^{-s}\frac{ds}{s} \leq \frac{C}{\Gamma(\beta)}\left\|\frac{\partial^{k}P_{t}f}{\partial
t^{k}}\right\|_{p(\cdot),\gamma_{d}}\int_{0}^{t}s^{\beta-1}ds\\
&\leq&\frac{C}{\Gamma(\beta)}\frac{t^{\beta}}{\beta}A_{k}(f)t^{-k+\alpha}
=C_{\beta} A_{k}(f)t^{-k+\alpha+\beta} .
\end{eqnarray*}
On the other hand, as $k>\alpha+\beta$  using Lemma
\ref{kdecay} as $t+s>s$, and since $f\in
B_{p(\cdot),\infty}^{\alpha}(\gamma_{d})$
\begin{eqnarray*}
(II) &\leq&\frac{C}{\Gamma(\beta)}\int_{t}^{\infty}s^{\beta}e^{-s}\left\|\frac{\partial^{k}P_{s}f}{\partial
s^{k}}\right\|_{p(\cdot),\gamma_{d}}\frac{ds}{s}
\leq C\frac{A_{k}(f)}{\Gamma(\beta)}\int_{t}^{\infty}s^{\beta}e^{-s} s^{-k+\alpha}\frac{ds}{s}\\
&\leq& C\frac{A_{k}(f)}{\Gamma(\beta)}\int_{t}^{\infty}s^{-k+\alpha+\beta-1}ds
=C\frac{A_{k}(f)}{\Gamma(\beta)}\frac{t^{-k+\alpha+\beta}}{k-(\alpha+\beta)} = C_{k,\alpha,\beta}A_{k}(f) t^{-k+\alpha+\beta}.
\end{eqnarray*}
Therefore,
\begin{eqnarray*}
\left\|\frac{\partial^{k}}{\partial
t^{k}} P_{t}({\mathcal{J}}_{\beta}f)\right\|_{p(\cdot),\gamma_{d}}&\leq&C A_{k}(f) t^{-k+\alpha+\beta}, \forall t>0.
\end{eqnarray*}
 Then
$\mathcal{J}_{\beta}f\in B_{p(\cdot),\infty}^{\alpha+\beta}(\gamma_{d})$
and $A_{k}(\mathcal{J}_{\beta}f)\leq CA_{k}(f)$. Thus,
\begin{eqnarray*}
\|{\mathcal{J}}_{\beta}f\|_{B_{p(\cdot),\infty}^{\alpha+\beta}}&=&\|{\mathcal{J}}_{\beta}f\|_{p(\cdot),\gamma_{d}}+A_{k}({\mathcal{J}}_{\beta}f)\\
&\leq&C\|f\|_{p(\cdot),\gamma_{d}}+C A_{k}(f) \leq C\|f\|_{B_{p(\cdot),\infty}^{\alpha}}. 
\end{eqnarray*}
\end{proof}

Now, in the following theorem we consider the action of Gaussian Bessel potentials on $B_{p(\cdot),q(\cdot)}^{\alpha}(\gamma_{d})$ spaces. It is analogous to Theorem 2.4 (i) of \cite{Pinurb}.
\begin{teo}
Let $\alpha\geq 0$, $\beta>0$, $p(\cdot)\in\mathcal{P}_{\gamma_{d}}^{\infty}(\mathbb{R}^{d})\cap LH_{0}(\mathbb{R}^{d})$ with $1< p_{-}\leq p_{+} <\infty$ and $q(\cdot)\in\mathcal{P}_{0,\infty}$. Then, the Gaussian Bessel potential
$\mathcal{J}_{\beta}$ is
bounded from $B_{p(\cdot),q(\cdot)}^{\alpha}(\gamma_{d})$ into
$B_{p(\cdot),q(\cdot)}^{\alpha+\beta}(\gamma_{d})$.
\end{teo}
\begin{proof} Let $f\in B_{p(\cdot),q(\cdot)}^{\alpha}(\gamma_{d})$  then $\mathcal{J}_{\beta}f\in L^{p(\cdot)}(\gamma_{d})$ since $\mathcal{J}_{\beta}$ is bounded on $L^{p(\cdot)}(\gamma_{d})$.\\
Let denote $u(x,t)=P_{t}f(x)$ and
$U(x,t)=P_{t}{\mathcal{J}}_{\beta}f(x)$. Using the representation (\ref{poissonrepstable}) of $P_t$, we have
$$U(x,t)=\displaystyle\int_{0}^{+\infty}T_{s}({\mathcal{J}}_{\beta}f)(x)\mu_{t}^{(1/2)}(ds) \quad .$$
Thus, by the semigroup's property of $P_{t}$
$$U(x,t_{1}+t_{2})=P_{t_{1}}(P_{t_{2}}({\mathcal{J}}_{\beta}f))(x)=
\displaystyle\int_{0}^{+\infty}T_{s}(P_{t_{2}}({\mathcal{J}}_{\beta}f))(x)\mu_{t_{1}}^{\frac{1}{2}}(ds).$$
Now, fix $k$ and $l$ integer greater than $\alpha$ and $\beta$
respectively. By using the dominated convergence theorem, differentiating $k$ times respect to $t_{2}$ and $l$ times
respect to $t_{1}$ we get

$$\displaystyle\frac{\partial^{k+l}U(x,t_{1}+t_{2})}{\partial
(t_{1}+t_{2})^{k+l}}=\int_{0}^{+\infty}T_{s}
(\frac{\partial^{k}P_{t_{2}}}{\partial
t_{2}^{k}}({\mathcal{J}}_{\beta}f))(x)\frac{\partial^{l}}{\partial
t_{1}^{l}}\mu_{t_{1}}^{\frac{1}{2}}(ds).$$ Thus, making $t=t_{1}+t_{2}$, we get

$$\displaystyle\frac{\partial^{k+l} U(x,t)}{\partial
t^{k+l}}=\int_{0}^{+\infty}T_{s}(\frac{\partial^{k}P_{t_{2}}}{\partial
t_{2}^{k}}({\mathcal{J}}_{\beta}f))(x)\frac{\partial^{l}}{\partial
t_{1}^{l}}\mu_{t_{1}}^{\frac{1}{2}}(ds),$$
therefore, by using Minkowski's integral inequality (\ref{integralMinkowski}), the $L^{p(\cdot)}$-continuity of $T_s$ and Lemma \ref{corol1}
\begin{eqnarray} \label{ineqU}
\nonumber \left\|\frac{\partial^{k+l} U(\cdot,t)}{\partial
t^{k+l}}\right\|_{p(\cdot),\gamma_{d}}&\leq&C
\int_{0}^{+\infty} \left\|T_{s} (\frac{\partial^{k}P_{t_{2}}}{\partial t_{2}^{k}}({\mathcal{J}}_{\beta}f))\right\|_{p(\cdot),\gamma_{d}} \left|\frac{\partial^{l}}{\partial t_{1}^{l}} \mu_{t_{1}}^{\frac{1}{2}}(ds)\right|\\
\nonumber &\leq&C\int_{0}^{+\infty} \left\|
\frac{\partial^{k}P_{t_{2}}}{\partial
t_{2}^{k}}({\mathcal{J}}_{\beta}f)\right\|_{p(\cdot),\gamma_{d}}\left|\frac{\partial^{l}}{\partial
t_{1}^{l}} \mu_{t_{1}}^{\frac{1}{2}}(ds)\right|\\
\nonumber &=&C \left\|\frac{\partial^{k}P_{t_{2}}}{\partial t_{2}^{k}}({\mathcal{J}}_{\beta}f)\right\|_{p(\cdot),\gamma_{d}}\int_{0}^{+\infty} \left|\frac{\partial^{l}}{\partial t_{1}^{l}} \mu_{t_{1}}^{\frac{1}{2}}(ds)\right|\\
&\leq& C (t_{1})^{-l}  \left\|\frac{\partial^{k}}{\partial
t_{2}^{k}}P_{t_{2}}{\mathcal{J}}_{\beta}f\right\|_{p(\cdot),\gamma_{d}} \quad .
\end{eqnarray}

On the other hand, using the representation of the Bessel potential
(\ref{Beselrepre}), we have

$$P_{t}({\mathcal{J}}_{\beta}f)(x)=\displaystyle\frac{1}{\Gamma(\beta)}\int_{0}^{+\infty}s^{\beta}e^{-s}P_{t+s}f(x)\frac{ds}{s}$$

Thus,
\begin{eqnarray*}
\frac{\partial^{k}P_{t}}{\partial t^{k}}({\mathcal{J}}_{\beta}f)(x)&=&\frac{1}{\Gamma(\beta)}\int_{0}^{+\infty}s^{\beta}e^{-s}\frac{\partial^{k} P_{t+s}f(x)}{\partial t^{k}}\frac{ds}{s}=\frac{1}{\Gamma(\beta)}\int_{0}^{+\infty}s^{\beta}e^{-s}\frac{\partial^{k}
P_{t+s}f(x)}{\partial (t+s)^{k}}\frac{ds}{s},
\end{eqnarray*}
and again by Minkowski's integral inequality (\ref{integralMinkowski})
\begin{eqnarray*}
\left\|\frac{\partial^{k} P_{t}}{\partial
t^{k}}({\mathcal{J}}_{\beta}f)\right\|_{p(\cdot),\gamma_{d}}&\leq&\frac{C}{\Gamma(\beta)}\int_{0}^{+\infty}s^{\beta}e^{-s}\left\|\frac{\partial^{k}P_{t+s}f}{\partial
(t+s)^{k}}\right\|_{p(\cdot),\gamma_{d}}\frac{ds}{s}.
\end{eqnarray*}
 Now, since the definition of $B_{p(\cdot),q(\cdot)}^{\alpha}(\gamma_{d})$ is
independent of the integer $k>\alpha$ that we choose, take $k>\alpha+\beta$ and $l>\beta$, then
$k+l>\alpha+2\beta>\alpha+\beta$, this is, $k+l$ is an integer greater than $\alpha+\beta$. Now we will show that
$$\left\|t^{k+l-(\alpha+\beta)}\left\|\frac{\partial^{k+l} U(\cdot,t)}{\partial
t^{k+l}}\right\|_{p(\cdot),\gamma_{d}}\right\|_{q(\cdot),\mu}<+\infty \quad .$$
 In fact, taking $t_1=t_2 = t/2$ in (\ref{ineqU}), we get
\begin{eqnarray*}
&&\left\|t^{k+l-(\alpha+\beta)}\left\|\frac{\partial^{k+l}
U(\cdot,t)}{\partial
t^{k+l}}\right\|_{p(\cdot),\gamma_{d}}\right\|_{q(\cdot),\mu}
\quad \quad\quad \quad
\quad\\
&&\hspace{2cm}\leq C
\left\|t^{k+l-(\alpha+\beta)}\left\|\frac{\partial^{k}P_{\frac{t}{2}}}{\partial
(\frac{t}{2})^{k}}({\mathcal{J}}_{\beta}f)\right\|_{p(\cdot),\gamma_{d}}(\frac{t}{2})^{-l}\right\|_{q(\cdot),\mu}\\
&&\hspace{2cm}\leq
\frac{C}{\Gamma(\beta)}\left\|t^{k-(\alpha+\beta)}\left(\int_{0}^{+\infty}s^{\beta}e^{-s}
\left\|\frac{\partial^{k}P_{s+\frac{t}{2}}f}{\partial(s+\frac{t}{2})^{k}}\right\|_{p(\cdot),\gamma_{d}}\frac{ds}{s}\right)\right\|_{q(\cdot),\mu}\\
&&\hspace{2cm}=\frac{C}{\Gamma(\beta)}\left\|t^{k-(\alpha+\beta)}\left(\int_{0}^{t}s^{\beta}
\left\|\frac{\partial^{k}P_{s+\frac{t}{2}}f}{\partial(s+\frac{t}{2})^{k}}\right\|_{p(\cdot),\gamma_{d}}\frac{ds}{s}\right)\right.\\
&&\hspace{4cm}+ t^{k-(\alpha+\beta)}\left(\int_{t}^{+\infty}s^{\beta}
\left.\left\|\frac{\partial^{k}P_{s+\frac{t}{2}}f}{\partial(s+\frac{t}{2})^{k}}\right\|_{p(\cdot),\gamma_{d}}\frac{ds}{s}\right)\right\|_{q(\cdot),\mu}\\
&&\hspace{2cm} \leq\frac{C}{\Gamma(\beta)}\left\|t^{k-(\alpha+\beta)}\left(\int_{0}^{t}s^{\beta}
\left\|\frac{\partial^{k}P_{s+\frac{t}{2}}f}{\partial(s+\frac{t}{2})^{k}}\right\|_{p(\cdot),\gamma_{d}}\frac{ds}{s}\right)\right\|_{q(\cdot),\mu}\\
&&\hspace{4cm} +\frac{C}{\Gamma(\beta)}\left\|t^{k-(\alpha+\beta)}
\left(\int_{t}^{+\infty}s^{\beta}
\left\|\frac{\partial^{k}P_{s+\frac{t}{2}}f}{\partial(s+\frac{t}{2})^{k}}\right\|_{p(\cdot),\gamma_{d}}\frac{ds}{s}\right)\right\|_{q(\cdot),\mu}\\
&&\hspace{2cm}=I+II.
\end{eqnarray*}
Using Lemma \ref{kdecay}, the change of variables $u=t/2$ and since $\beta>0$, we have
\begin{eqnarray*}
I&=&\frac{C}{\Gamma(\beta)}\left\|t^{k-(\alpha+\beta)}\left(\int_{0}^{t}s^{\beta}
\left\|\frac{\partial^{k}P_{s+\frac{t}{2}}f}{\partial(s+\frac{t}{2})^{k}}\right\|_{p(\cdot),\gamma_{d}}\frac{ds}{s}\right)\right\|_{q(\cdot),\mu}\\
&\leq&\frac{C}{\Gamma(\beta)}\left\|t^{k-(\alpha+\beta)}\left(\int_{0}^{t}s^{\beta}
\left\|\frac{\partial^{k}P_{\frac{t}{2}}f}{\partial(\frac{t}{2})^{k}}\right\|_{p(\cdot),\gamma_{d}}\frac{ds}{s}\right)\right\|_{q(\cdot),\mu}\\
&=&\frac{C}{\beta\Gamma(\beta)}\left\|t^{k-\alpha}\left\|\frac{\partial^{k}P_{\frac{t}{2}}f}{\partial(\frac{t}{2})^{k}}\right\|_{p(\cdot),\gamma_{d}}\right\|_{q(\cdot),\mu}=C_{k,\alpha,\beta}\left\|u^{k-\alpha}\left\|\frac{\partial^{k}P_{u}f}{\partial
u^{k}}\right\|_{p(\cdot),\gamma_{d}}\right\|_{q(\cdot),\mu}<+\infty,
\end{eqnarray*}
since $f\in B_{p(\cdot),q(\cdot)}^{\alpha}(\gamma_{d})$.\\

On the other hand, using the Hardy's ineguality (\ref{hardyineqtainftyr}), since
$k>\alpha+\beta$ and  again  by Lemma \ref{kdecay}, we get
\begin{eqnarray*}
II&=&\frac{C}{\Gamma(\beta)}\left\|t^{k-(\alpha+\beta)}\left(\int_{t}^{+\infty}s^{\beta}
\left\|\frac{\partial^{k}P_{s+\frac{t}{2}}f}{\partial(s+\frac{t}{2})^{k}}\right\|_{p(\cdot),\gamma_{d}}\frac{ds}{s}\right)\right\|_{q(\cdot),\mu}\\
&\leq&\frac{C}{\Gamma(\beta)}\left\|t^{k-(\alpha+\beta)}\left(\int_{t}^{+\infty}s^{\beta}
\left\|\frac{\partial^{k}P_{s}f}{\partial s^{k}}\right\|_{p(\cdot),\gamma_{d}}\frac{ds}{s}\right)\right\|_{q(\cdot),\mu} \leq C_{k,\alpha,\beta}\left\|s^{k-\alpha}\left\|\frac{\partial^{k}}{\partial
s^{k}}P_{s}f\right\|_{p(\cdot),\gamma_{d}}\right\|_{q(\cdot),\mu}<+\infty,
\end{eqnarray*}
since $f\in B_{p(\cdot),q(\cdot)}^{\alpha}(\gamma_{d})$. This is,
${\mathcal{J}}_{\beta}f\in B_{p(\cdot),q(\cdot)}^{\alpha+\beta}(\gamma_{d})$.\\

Moreover,

  $$\displaystyle\|{\mathcal{J}}_{\beta}f\|_{B_{p(\cdot),q(\cdot)}^{\alpha+\beta}}\leq
C \|f\|_{B_{p(\cdot),q(\cdot)}^{\alpha}}.$$
\end{proof}

Now, we will study  the action of  Bessel fractional derivative $\mathcal{D}^\beta$ on variable Gaussian Besov-Lipschitz spaces $B_{p(\cdot),q(\cdot)}^{\alpha}(\gamma_{d})$ . We will use the representation (\ref{derbesselrep}) of the Bessel fractional derivative and Hardy's inequalities.\\

First, we need to consider the forward differences. Remember for a given function $f$,  the $k$-th order forward difference of $f$ starting at $t$ with increment $s$ is defined as,\\
 $$\Delta_{s}^{k}(f,t)=\displaystyle\sum_{j=0}^{k}{k\choose j}(-1)^{j}f(t+(k-j)s).$$
 The forward differences have the following properties (see Appendix 10.9 in \cite{urbina2019})
 we will need the following technical result
 \begin{lem}\label{forw-diff} For any positive integer $k$
 \begin{enumerate}
 \item[i)]$\Delta_{s}^{k}(f,t)=\Delta_{s}^{k-1}(\Delta_{s}(f,\cdot),t)=\Delta_{s}(\Delta_{s}^{k-1}(f,\cdot),t)$
 \item[ii)] $\Delta_{s}^{k}(f,t)=\displaystyle\int_{t}^{t+s}\int_{v_{1}}^{v_{1}+s}...
 \int_{v_{k-2}}^{v_{k-2}+s}\int_{v_{k-1}}^{v_{k-1}+s}f^{(k)}(v_{k})dv_{k}dv_{k-1}...dv_{2}dv_{1}$.\\
For any positive integer $k$,
\begin{equation}\label{difder}
\frac{\partial }{\partial s}(\Delta_s^k (f,t) ) = k \,\Delta_s^{k-1} (f',t+s),
\end{equation}
and for any integer $j>0$,
\begin{equation}\label{difder2}
\frac{\partial^j }{\partial t^j}(\Delta_s^k (f,t) ) =\Delta_s^{k} (f^{(j)},t).
\end{equation}
 \end{enumerate}
 \end{lem}

 Observe that, using the Binomial Theorem and the semigroup property of $\{P_t\}$, we have
 \begin{eqnarray}\label{powerrep}
\nonumber ( P_t -I )^k f(x) &=& \sum_{j=0}^k {k \choose j} P_t^{k-j} (-I)^j f(x) = \sum_{j=0}^k {k \choose j} (-1)^jP_t^{k-j} f(x)\\
\nonumber &=&\sum_{j=0}^k {k \choose j} (-1)^jP_{(k-j)t} f(x) =\sum_{j=0}^k {k \choose j} (-1)^j u(x,(k-j)t)\\
&=& \Delta_t^k (u(x, \cdot), 0),
\end{eqnarray}
where  as usual, $u(x,t) = P_t f(x)$.\\

 %====================desigualdad en norma para forw-diff
Additionally, we will need in what follows the next result,
 \begin{lem}\label{lpforw-diff}
Let $p(\cdot)\in \mathcal{P}(\mathbb{R}^{d},\gamma_{d})$, $f\in L^{p(\cdot)}(\gamma_{d})$ and $k,n\in\mathbb{N}$ then
$$\displaystyle\|\Delta_{s}^{k}(u^{(n)},t)\|_{p(\cdot),\gamma_{d}}\leq C_{k,p(\cdot)}
s^{k}\|u^{(k+n)}(\cdot,t)\|_{p(\cdot),\gamma_{d}} $$
 \end{lem}
 \begin{proof}
 From ii) of Lemma \ref{forw-diff}, we have
  $$\Delta_{s}^{k}(u^{(n)}(x,\cdot),t)=\displaystyle\int_{t}^{t+s}\int_{v_{1}}^{v_{1}+s}...\int_{v_{k-2}}^{v_{k-2}+s}
 \int_{v_{k-1}}^{v_{k-1}+s}u^{(k+n)}(x,v_{k})dv_{k}dv_{k-1}...dv_{2}dv_{1},$$
 then, using Minkowski's integral inequality (\ref{integralMinkowski}) and Lemma \ref{kdecay} $k$-times respectively
 \begin{eqnarray*}
\|\Delta_{s}^{k}(u^{(n)},t)\|_{p(\cdot),\gamma_{d}}&\leq&C^{k}\displaystyle\int_{t}^{t+s}\int_{v_{1}}^{v_{1}+s}...
\int_{v_{k-2}}^{v_{k-2}+s}\int_{v_{k-1}}^{v_{k-1}+s}\|u^{(k+n)}(\cdot,v_{k})\|_{p(\cdot),\gamma_{d}}dv_{k}dv_{k-1}...dv_{2}dv_{1}\\
&\leq&C^{k}(C_{p(\cdot)})^{k}s^{k}\|u^{(k+n)}(\cdot,t)\|_{p(\cdot),\gamma_{d}}=C_{k,p(\cdot)}s^{k}\left\|\frac{\partial^{k+n}}{\partial
 t^{k+n}}u(\cdot,t)\right\|_{p(\cdot),\gamma_{d}}. 
 \end{eqnarray*}
\end{proof}
%================caso alfa mayor o igual a 1

%===========================================derivada Bessel
We are now ready to  consider  the action Gaussian Bessel fractional derivatives on general $B_{p(\cdot),q(\cdot)}^{\alpha}(\gamma_{d})$ spaces. The result analogous to Theorem 8 in \cite{gatPinurb}.
\begin{teo}\label{DerBessel>1}
Let $0<\beta<\alpha$, $p(\cdot)\in\mathcal{P}_{\gamma_{d}}^{\infty}(\mathbb{R}^{d})\cap LH_{0}(\mathbb{R}^{d})$ with $1< p_{-}\leq p_{+} <\infty$ and $q(\cdot)\in\mathcal{P}_{0,\infty}$. Then, the Gaussian Bessel fractional derivative $\mathcal{D}_{\beta}$ is
bounded from $B_{p(\cdot),q(\cdot)}^{\alpha}(\gamma_{d})$ into
$B_{p(\cdot),q(\cdot)}^{\alpha-\beta}(\gamma_{d})$.
\end{teo}
\begin{proof}

Let $f\in B_{p(\cdot),q(\cdot)}^{\alpha}(\gamma_{d})$, $k\in \mathbb{N}$ such that $k-1\leq \beta<k$ and set $v(x,t)=e^{-t} u(x,t)$ then using the classical Hardy's inequality (\ref{hardy1}), the fundamental theorem of calculus and Lemma \ref{forw-diff},
 \begin{eqnarray*}
 \nonumber |\mathcal{D}_{\beta}f(x)|&\leq&\displaystyle\frac{1}{c_{\beta}}\int_{0}^{+\infty}s^{-\beta-1}|\Delta_{s}^{k}(v(x,\cdot),0)|ds\\
&\leq&\displaystyle\frac{1}{c_{\beta}}\int_{0}^{+\infty}s^{-\beta-1}\int_{0}^{s}\left|\frac{\partial}{\partial r}\Delta_{r}^{k}(v(x,\cdot),0)\right|dr\,ds\leq\displaystyle\frac{k}{\beta c_{\beta}}\int_{0}^{+\infty}r^{-\beta}|\Delta_{r}^{k-1}(v'(x,\cdot),r)|dr
\end{eqnarray*}
and by Minkowski's integral inequality (\ref{integralMinkowski}) this implies 
$$\|\mathcal{D}_{\beta}f\|_{p(\cdot),\gamma_{d}}\leq\displaystyle\frac{k}{\beta c_{\beta}}C
\int_{0}^{+\infty}r^{-\beta}\|\Delta_{r}^{k-1}(v',r)\|_{p(\cdot),\gamma_{d}}dr.$$
Now, using Lemma \ref{forw-diff} and again Minkowski's integral inequality (\ref{integralMinkowski})

$\|\Delta_{r}^{k-1}(v',r)\|_{p(\cdot),\gamma_{d}}\leq C^{k}\displaystyle\int_{r}^{2r}\int_{v_{1}}^{v_{1}+r}...
 \int_{v_{k-2}}^{v_{k-2}+r}\|v^{(k)}(\cdot,v_{k-1})\|_{p(\cdot),\gamma_{d}}dv_{k-1}dv_{k-2}...dv_{2}dv_{1},$\\
and by Leibnitz's differentiation rule for the product
 \begin{eqnarray*}
\|v^{(k)}(\cdot,v_{k-1})\|_{p(\cdot),\gamma_{d}}&=&\displaystyle\left\|\sum_{j=0}^{k}{k\choose j}(e^{-v_{k-1}})^{(j)}u^{(k-j)}(\cdot,v_{k-1})\right\|_{p(\cdot),\gamma_{d}}\\
&\leq&\displaystyle\sum_{j=0}^{k}{k\choose j}e^{-v_{k-1}}\|u^{(k-j)}(\cdot,v_{k-1})\|_{p(\cdot),\gamma_{d}}.
  \end{eqnarray*}
  Then, by Lemma \ref{kdecay}
  \begin{eqnarray*}
&& \|\Delta_{r}^{k-1}(v',r)\|_{p(\cdot),\gamma_{d}}\\
&& \quad \quad \quad \leq C^{k}\displaystyle\sum_{j=0}^{k}{k\choose j}\int_{r}^{2r}\int_{v_{1}}^{v_{1}+r}...
 \int_{v_{k-2}}^{v_{k-2}+r}e^{-v_{k-1}}\|u^{(k-j)}(\cdot,v_{k-1})\|_{p(\cdot),\gamma_{d}}dv_{k-1}dv_{k-2}...dv_{2}dv_{1}\\
 &&\quad \quad \quad \leq C_{k,p(\cdot)} \sum_{j=0}^{k}{k\choose j}r^{k-1}e^{-r}\|u^{(k-j)}(\cdot,r)\|_{p(\cdot),\gamma_{d}}.
 \end{eqnarray*}
Therefore, using the $L^{p(\cdot)}$-boundedness of $P_{t}$ (see \cite{MorPinUrb})
\begin{eqnarray*}
 \nonumber \|\mathcal{D}_{\beta}f\|_{p(\cdot),\gamma_{d}}&\leq&\displaystyle\frac{k}{\beta c_{\beta}}C_{k,p(\cdot)}
\sum_{j=0}^{k}{k\choose j}\int_{0}^{+\infty}
 r^{k-\beta-1}e^{-r}\|u^{(k-j)}(\cdot,r)\|_{p(\cdot),\gamma_{d}}dr\\
&=&C_{k,p(\cdot)}\displaystyle\frac{k}{\beta c_{\beta}}\sum_{j=0}^{k-1}{k\choose j}\int_{0}^{+\infty}
 r^{(k-j)-(\beta-j)-1}e^{-r}\left\|\frac{\partial^{k-j}}{\partial r^{k-j}}P_{r}f\right\|_{p(\cdot),\gamma_{d}}dr\\
&&\hspace{2.5cm} +\;C_{k,p(\cdot)}\frac{k}{\beta c_{\beta}}\int_{0}^{+\infty}
 r^{k-\beta-1}e^{-r}\|P_{r}f\|_{p(\cdot),\gamma_{d}}dr\\
 &\leq&C_{k,p(\cdot)}\displaystyle\frac{k}{\beta c_{\beta}}\sum_{j=0}^{k-1}{k\choose j}\int_{0}^{+\infty}
 r^{(k-j)-(\beta-j)-1}\left\|\frac{\partial^{k-j}}{\partial r^{k-j}}P_{r}f\right\|_{p(\cdot),\gamma_{d}}dr\\
&& \hspace{2.5cm} + \;C_{k,p(\cdot)}\frac{k}{\beta c_{\beta}}\int_{0}^{+\infty}
 r^{k-\beta-1}e^{-r}\|f\|_{p(\cdot),\gamma_{d}}dr
\end{eqnarray*}
 Thus,
\begin{eqnarray*}
 \nonumber \|\mathcal{D}_{\beta}f\|_{p(\cdot),\gamma_{d}}&\leq&C_{k,p(\cdot)}\displaystyle\frac{k}{\beta c_{\beta}}\sum_{j=0}^{k-1}{k\choose j}\int_{0}^{+\infty}
 r^{k-j-(\beta-j)}\left\|\frac{\partial^{k-j}}{\partial r^{k-j}}P_{r}f\right\|_{p(\cdot),\gamma_{d}}\frac{dr}{r}\\
&&\hspace{2.5cm} + C_{k,p(\cdot)}\frac{k \Gamma(k-\beta)}{\beta c_{\beta}}\|f\|_{p(\cdot),\gamma_{d}}<\infty,
\end{eqnarray*}
 since $f\in B_{p(\cdot),q(\cdot)}^{\alpha}(\gamma_{d})\subset B_{p(\cdot),1}^{\beta-j}(\gamma_d) $ as $\alpha>\beta>\beta-j\geq 0$, for $j\in\{0,...,k-1\}$. Hence, $\mathcal{D}_{\beta}f\in L^{p(\cdot)}(\gamma_d)$.\\

%===============condici'on Besov
On the other hand,
 $$P_{t}(e^{-s}P_{s}-I)^{k}f(x)=\displaystyle\sum_{j=0}^{k}{k\choose j}(-1)^{j}e^{-s(k-j)}u(x,t+(k-j)s).$$
Let $n$ be the smaller integer greater than $\alpha$, i.e. $ n-1 \leq \alpha  < n$, we have
 \begin{eqnarray*}
 \frac{\partial^{n}}{\partial t^{n}}P_{t}(\mathcal{D}_{\beta}f)(x)&=&\frac{1}{c_{\beta}}
 \int_{0}^{+\infty}s^{-\beta-1}\sum_{j=0}^{k}{k\choose j}(-1)^{j}e^{-s(k-j)}u^{(n)}(x,t+(k-j)s)ds\\
 &=&\frac{e^t}{c_{\beta}}
 \int_{0}^{+\infty}s^{-\beta-1}\sum_{j=0}^{k}{k\choose j}(-1)^{j}e^{-(t+s(k-j))}u^{(n)}(x,t+(k-j)s)ds\\
&=&\frac{e^{t}}{c_{\beta}}
 \int_{0}^{+\infty}s^{-\beta-1}\Delta_{s}^{k}(w(x,\cdot),t) ds,
  \end{eqnarray*}
where $w(x,t)=\displaystyle e^{-t}u^{(n)}(x,t)$. Now using the fundamental theorem of calculus,
 \begin{eqnarray*}
 \frac{\partial^{n}}{\partial t^{n}}P_{t}(\mathcal{D}_{\beta}f)(x)&=&\frac{e^{t}}{c_{\beta}}
 \int_{0}^{+\infty}s^{-\beta-1}\Delta_{s}^{k}(w(x,\cdot),t)ds\\
&=&\frac{e^{t}}{c_{\beta}}
  \int_{0}^{+\infty}s^{-\beta-1}\int_{0}^{s}\frac{\partial}{\partial
  r}\Delta_{r}^{k}(w(x,\cdot),t)dr\,ds.
  \end{eqnarray*}
Then, using classical Hardy's inequality (\ref{hardy1}), and Lemma \ref{forw-diff},
 \begin{eqnarray*}
 \left|\frac{\partial^{n}}{\partial t^{n}}P_{t}\mathcal({D}_{\beta}f)(x)\right|&\leq&\frac{e^{t}}{c_{\beta}}
  \int_{0}^{+\infty}s^{-\beta-1}\int_{0}^{s} \left|\frac{\partial}{\partial
  r}\Delta_{r}^{k}(w(x,\cdot),t) \right|drds\\
&\leq&\frac{e^{t}}{c_{\beta}\beta}
  \int_{0}^{+\infty}r \left|\frac{\partial}{\partial
  r}\Delta_{r}^{k}(w(x,\cdot),t)\right|r^{-\beta-1}dr\\
  &=&\frac{ke^{t}}{c_{\beta}\beta}
  \int_{0}^{+\infty}r^{-\beta}|\Delta_{r}^{k-1}(w'(x,\cdot),t+r)|dr
  \end{eqnarray*}
  and by Minkowski's integral inequality (\ref{integralMinkowski}) we get
 \begin{eqnarray*}
 \left\|\frac{\partial^{n}}{\partial t^{n}}P_{t}(\mathcal{D}_{\beta}f)\right\|_{p(\cdot),\gamma_{d}}&\leq&C\frac{ke^{t}}{\beta c_{\beta}}
  \int_{0}^{+\infty}r^{-\beta}\|\Delta_{r}^{k-1}(w',t+r)\|_{p(\cdot),\gamma_{d}}dr.
  \end{eqnarray*}
  Now, by analogous argument as above, Lemma \ref{forw-diff} and again Leibnitz's differentiation rule for the product, give us
 \begin{eqnarray*}
\|\Delta_{r}^{k-1}(w',t+r)\|_{p(\cdot),\gamma_{d}}&\leq&C_{k,p(\cdot)}\sum_{j=0}^{k}{k\choose j}r^{k-1}e^{-(t+r)}\|u^{(k+n-j)}(\cdot,t+r)\|_{p(\cdot),\gamma_{d}},
\end{eqnarray*}
and this implies that
 \begin{eqnarray*}\label{pest4}
 \left\|\frac{\partial^{n}}{\partial t^{n}}P_{t}(\mathcal{D}_{\beta}f)\right\|_{p(\cdot),\gamma_{d}}&\leq&C_{k,p(\cdot)}e^{t}\frac{k}{c_{\beta}\beta}
  \int_{0}^{+\infty}r^{-\beta}\big(\sum_{j=0}^{k}{k\choose j}r^{k-1}e^{-(t+r)}\|u^{(k+n-j)}(\cdot,t+r)\|_{p(\cdot),\gamma_{d}}
dr\\
 &=&C_{k,p(\cdot)}\frac{k}{c_{\beta}\beta}
  \sum_{j=0}^{k}{k\choose j}\int_{0}^{+\infty}r^{k-\beta-1}e^{-r}\|u^{(k+n-j)}(\cdot,t+r)\|_{p(\cdot),\gamma_{d}}
dr.
\end{eqnarray*}
Thus,
 \begin{eqnarray*}
 &&\left\|t^{n-(\alpha-\beta)}\left\|\frac{\partial^{n}}{\partial t^{n}}P_{t}(\mathcal{D}_{\beta}f)\right\|_{p(\cdot),\gamma_{d}}\right\|_{q(\cdot),\mu}\\
 &&\quad \quad \leq C_{k,p(\cdot)}\frac{k}{c_{\beta}\beta}
  \sum_{j=0}^{k}{k\choose j} \left\|t^{n-(\alpha-\beta)}\int_{0}^{+\infty}r^{k-\beta-1}e^{-r}\|u^{(k+n-j)}(\cdot,t+r)\|_{p(\cdot),\gamma_{d}}
dr\right\|_{q(\cdot),\mu}.
\end{eqnarray*}
Now, for each $1\leq j\leq k$,
$0<\alpha-\beta+k-j\leq\alpha$ and  by lemma  \ref{kdecay}
\begin{eqnarray*}
&&
\left\|t^{n-(\alpha-\beta)}\int_{0}^{\infty}r^{k-\beta-1}e^{-r}\|u^{(k+n-j)}(\cdot,t+r)\|_{p(\cdot),\gamma_{d}}
dr\right\|_{q(\cdot),\mu}\\
&& \quad \quad \quad  \leq  C_{p(\cdot)}
\left \|t^{n-(\alpha-\beta)}\|u^{(n+k-j)}(\cdot,t)\|_{p(\cdot),\gamma_{d}}\int_{0}^{+\infty}r^{k-\beta-1}e^{-r}
dr\right\|_{q(\cdot),\mu}\\
&& \quad \quad \quad  = C_{p(\cdot)} \Gamma(k-\beta)
\left\|t^{n+(k-j)-(\alpha-\beta+k-j)}\|u^{(n+k-j)}(\cdot,t)\|_{p(\cdot),\gamma_{d}}\right\|_{q(\cdot),\mu}<\infty,
\end{eqnarray*}
as  $f\in B_{p(\cdot),q(\cdot)}^{\alpha}(\gamma_{d})\subset
B_{p(\cdot),q(\cdot)}^{\alpha-\beta+(k-j)}(\gamma_{d})$ for any $1\leq j \leq k$.\\

Now, for the case $j=0$,
 \begin{eqnarray*}
&&\left\|t^{n-(\alpha-\beta)}\int_{0}^{+\infty}r^{k-\beta-1}e^{-r}\|u^{(n+k)}(\cdot,t+r)\|_{p(\cdot),\gamma_{d}}
dr\right\|_{q(\cdot),\mu}\\
&&\quad \quad \quad \leq\left\|t^{n-(\alpha-\beta)}\int_{0}^{t}r^{k-\beta-1}e^{-r}\|u^{(n+k)}(\cdot,t+r)\|_{p(\cdot),\gamma_{d}}
dr\right\|_{q(\cdot),\mu}\\
&&\quad \quad \quad \quad \quad +\left\|t^{n-(\alpha-\beta)}\int_{t}^{+\infty}r^{k-\beta-1}e^{-r}\|u^{(n+k)}(\cdot,t+r)\|_{p(\cdot),\gamma_{d}}
dr\right\|_{q(\cdot),\mu}\\
&&\quad \quad \quad =(I)+(II).
\end{eqnarray*}
Using Lemma \ref{kdecay}, and $k>\beta$,
 \begin{eqnarray*}
(I) &\leq&C_{p(\cdot)}\left\|t^{n-(\alpha-\beta)}\int_{0}^{t}r^{k-\beta-1}\|u^{(n+k)}(\cdot,t)\|_{p(\cdot),\gamma_{d}}
dr\right\|_{q(\cdot),\mu}\\
&=&C_{p(\cdot)}\left\|t^{n-(\alpha-\beta)}\|u^{(n+k)}(\cdot,t)\|_{p(\cdot),\gamma_{d}}\int_{0}^{t}r^{k-\beta-1}dr\right\|_{q(\cdot),\mu}\\
&=&\frac{C_{p(\cdot)}}{k-\beta}\left\|t^{n+k-\alpha}\|u^{(n+k)}(\cdot,t)\|_{p(\cdot),\gamma_{d}}\right\|_{q(\cdot),\mu}<\infty,
\end{eqnarray*}
since $f\in B_{p(\cdot),q(\cdot)}^{\alpha}(\gamma_{d})$ and $n+k>\alpha$.\\

For the second term, using Lemma \ref{kdecay} and Hardy's inequality  (\ref{hardyineqtainftyr})
 \begin{eqnarray*}
(II) &\leq&C_{p(\cdot)}\left\|t^{n-(\alpha-\beta)}\int_{t}^{+\infty}r^{k-\beta-1}\|u^{(n+k)}(\cdot,r)\|_{p(\cdot),\gamma_{d}}
dr\right\|_{q(\cdot),\mu}\\
&\leq&C_{p(\cdot)}C_{q(\cdot)}\left\|r^{n+k-\alpha}\|u^{(n+k)}(\cdot,r)\|_{p(\cdot),\gamma_{d}}\right\|_{q(\cdot),\mu}<\infty,
\end{eqnarray*}
since $f\in B_{p(\cdot),q(\cdot)}^{\alpha}(\gamma_{d})$. Therefore,  $\mathcal{D}_{\beta}f\in
B_{p(\cdot),q(\cdot)}^{\alpha-\beta}(\gamma_{d})$.\\

 Moreover,
 \begin{eqnarray*}
\|\mathcal{D}_{\beta}f\|_{B_{p(\cdot),q(\cdot)}^{\alpha-\beta}}&=&\|\mathcal{D}_{\beta}f\|_{p(\cdot),\gamma_{d}}+
\left\|t^{n-(\alpha-\beta)}\left\|\frac{\partial^{n}}{\partial t^{n}}P_{t}\mathcal{D}_{\beta}f\right\|_{p(\cdot),\gamma_{d}}\right\|_{q(\cdot),\mu}\\
 &\leq&C^{1}_{k,p(\cdot)}\|f\|_{p(\cdot),\gamma_{d}}+\frac{k}{c_{\beta}\beta}\sum_{j=0}^{k}{k\choose j} C^{2}_{p(\cdot),q(\cdot)}\left\|r^{n-\alpha}\left\|\frac{\partial^{n}}{\partial r^{n}}P_{r}f\right\|_{p(\cdot),\gamma_{d}}\right\|_{q(\cdot),\mu}\\
&\leq&C_{p(\cdot),q(\cdot)}\|f\|_{B_{p(\cdot),q(\cdot)}^{\alpha}}
\end{eqnarray*}
 \end{proof}

The boundedness of Gaussian Riesz potentials on variable Gaussian Besov-Lipschitz spaces and the regularity of all these operators on variable  Gaussian Triebel-Lizorkin spaces,
spaces that were also defined in \cite{Pinrodurb}, will be considered in a forthcoming paper.

%%======================================

%======================================

\end{document}